\documentclass[a4paper,12pt,openany,reqno]{amsart}
	\usepackage[latin1]{inputenc}
    \usepackage[T1]{fontenc}
    \usepackage[english]{babel}
    \usepackage{appendix}
    \usepackage[english]{minitoc}
		\usepackage[top=2.5cm, bottom=2.5cm, left=2.5cm, right=2.5cm]{geometry}
\numberwithin{equation}{section}
\makeatletter\@addtoreset{equation}{section} 
	\usepackage{amsmath}
    \usepackage{amssymb}
    \usepackage{amsbsy} 
    \usepackage{latexsym, amsfonts}
    \usepackage{graphics}
    \usepackage{ulem}
    \usepackage{hhline}
    \usepackage{dsfont}
    \usepackage{mathrsfs}
    \usepackage{color}
    \usepackage{fancyhdr}
    \usepackage{rotating}
    \usepackage{fancybox}
    \usepackage{colortbl}
    \usepackage{pifont}
    \usepackage{setspace}
    \usepackage{enumerate}
    \usepackage{multicol}
    \usepackage{varioref}
    \usepackage{lmodern}
    \usepackage{textcomp}
    \usepackage{euscript}
    \usepackage[pdftex]{hyperref}
			\newtheorem{theorem}{Theorem}[section]
			\newtheorem{definition}[theorem]{Definition}
			
			\newtheorem{lemma}[theorem]{Lemma}
			\newtheorem{proposition}[theorem]{Proposition}
			
			\newtheorem{corollary}[theorem]{Corollary}
			\newtheorem{remark}[theorem]{Remark}

\newcommand{\C}{\mathbb C}   
 \newcommand{\R}{\mathbb R}  
  \newcommand{\Z}{\mathbb Z} 	 
  \newcommand{\Hq}{\mathbb H}
   \newcommand{\Sq}{\mathbb S}
       \newcommand{\wSq}{\widetilde{\Sq}}
   \newcommand{\As}{|\wSq|}
    \newcommand{\Ch}{\mathcal{C}_{\Hq}}
     \newcommand{\Cc}{\mathcal{C}_{\C}}
   
\newcommand{\norm}[1]{{\left\|{#1}\right\|}}    
   \newcommand{\scal}[1]{{\left\langle{#1}\right\rangle}}
    \newcommand{\bz}{\overline{z}}

\newcommand{\srn}{\mathcal{SR}_{n}}

\newcommand{\LgH}{L^{2}_\mu(\Hq)}
\newcommand{\LgS}{L^{2}_{\mu_I}(\C_I)}
\newcommand{\Hnew}{\mathcal{H}^{2}_{slice}(\Hq)}
\newcommand{\bq}{\overline{q} }
\newcommand{\bp}{\overline{p}}
\newcommand{\bxi}{\overline{\xi}}
\makeatletter
\@namedef{subjclassname@2020}{%
	\textup{2020} Mathematics Subject Classification}
\makeatother
\begin{document}

\title{Weighted  quaternionic Cauchy singular integral}
\author{A. Elkachkouri}  
\author{A. Ghanmi}

\address{Analysis, P.D.E $\&$ Spectral Geometry, Lab M.I.A.-S.I., CeReMAR, 
	Department of Mathematics,  
	P.O. Box 1014,  Faculty of Sciences, 
	Mohammed V University in Rabat, Morocco}

\email{elkachkouri.abdelatif@gmail.com}

 \email{allal.ghanmi@um5.ac.ma/ag@fsr.ac.ma}

\begin{abstract}
We investigate some spectral properties of the weighted quaternionic Cauchy transform when acting on the right quaternionic Hilbert space of Gaussian integrable functions. We study its boundedness, compactness, and memberships to $k$-Schatten class, and we identify its
range. This is done by means of its restriction to the n-th S-polyregular Bargmann space of second kind, for which we provide an explicit closed expression for its action on the quaternionic It\^o--Hermite polynomials constituting an orthogonal basis. We also exhibit an orthogonal basis
of eigenfunctions of its n-Bergman projection leading to the explicit determination of its singular values. The obtained results generalize those given for weighted Cauchy transform on the complex plane to the quaternionic setting.
\end{abstract}



\maketitle
	
	\maketitle

\section{INTRODUCTION}

In \cite{BirkhoffvonNeumann1936}, Birkhoff and von Neumann showed that the Schr\"odinger equation can also be formulated using the quaternionic numbers. 
Since then, there was an increasing interest in studying and developing the quaternionic analog of quantum mechanics \cite{Adler1997,Emch1963,FinkelsteinJauchSchiminovichSpeiser1962}.
This motivated, in particular, the investigation of the spectral theory for quaternionic linear operators and the theory of quaternionic groups. 
However, one of the main obstacle to develop such a spectral theory was the lack of a precise notion of a quaternionic spectrum.
The discovery of the notion of S-spectrum, in 2006 by F. Colombo and I. Sabadini, allowed to fully develop of quaternionic operator theory  \cite{11ColomboSabadiniStruppa2011}, and of the intrinsic S-functional calculus (the quaternionic version of the Riesz--Dunford functional calculus).
A complete description of S-calculus based on the Cauchy formula and the slice hyperholomorphic kernels is given in \cite{8ColomboSabadini2009}.  See also the seminal papers \cite{4AlpayColomboGantnerSabadini2015,9ColomboSabadini2010} as well as \cite{10Gantner2020}. 
In fact, with the introduction of the new notion of slice regular functions on the quaternions \cite{GentiliStruppa07}, it is observed 
that the Cauchy formula holds on every slice, and its kernel is given by 
\begin{align}\label{CkC}
	 \sum_{n=0}^{\infty} q^n p^{-n-1} = \frac{1}{p-q},
\end{align}
valid for $p,q$ belonging to the same slice and such that $|q|< |p|$. The closed formula for this expansion series for arbitrary 
quaternions $p,q$; $|q|< |p|$, was given by F. Colombo and I. Sabadini, 
\begin{align}\label{CkH}  \sum_{n=0}^{\infty} q^n p^{-n-1} = -(q^2 - 2 q\Re (p) + |p|^2)^{-1}(q-\bp),
\end{align}
which leads to a closed formula of the formal series $ \sum_{n=0}^{\infty} T^n p^{-n-1}$, for given bounded linear operator $T$ on the division algebra of quaternions $\Hq$, and therefore to the natural notion of the S-spectrum defined as
$$ \sigma_S(T) := \{p \in \Hq; \, T^2 - 2 \Re(p) T + |p|^2Id \mbox{ is not invertible} \}.$$ 
Notice for instance that the point S-spectrum is exactly the set of right eigenvalues \cite{8ColomboSabadini2009,GhiloniMorettiPerotti2013}.

Such notion of S-spectrum had interesting consequences on the precise formulation of the spectral theorem for quaternionic linear operators. In particular, a full proof of this fundamental 
theorem for both bounded and unbounded normal operators has been given by D. Alpay, F. Colombo, D.P. Kimsey in \cite{5AlpayColomboKimsey2016}.
The developed quaternionic calculus leads to the study 
of evolution operators \cite{ColomboSabadini2011}, and the perturbation of normal quaternionic operators  and the existence of the associated invariant subspace \cite{3CerejeirasColomboKahlerSabadini2019}.
It also allows the development of Schur analysis in the context of slice hyperholomorphic  
\cite{2AlpayColomboSabadini2016}.  
The interested reader can find in \cite[Subsection 1.2.1]{6ColomboGantnerKimsey2018} a well explanation how hypercomplex analysis methods were used to identify the appropriate notion of quaternionic
spectrum.
For a recent overview of the spectral theory on the S-spectrum, its related problems and 
the most recent advances in its applications, we can consult
 \cite{6ColomboGantnerKimsey2018,7ColomboGantner2019,12ColomboGantnerPinton2021} and the references therein.  
For the spectral theory on the S-spectrum for Clifford operators, we refer the reader to \cite{11ColomboSabadiniStruppa2011,5AlpayColomboKimsey2016}.

On the other hand, associated with the classical Cauchy formula for the complex holomorphic functions, one considers the singular integral transform on the whole complex plane  defined by
\begin{align}\label{CT}
	\Cc f(z) :=  \frac 1\pi \int_{\C} \frac{f(\xi)}{z-\xi} e^{-|\xi|^2}dxdy ; \xi =x+iy, 
\end{align}
and acting on the Hilbert space $L^2_g(\C)=L^2(\C,e^{-|z|^2}dxdy)$ of gaussian functions.
The kernel function in \eqref{CT} is the Cauchy kernel function in the right hand side of \eqref{CkC}.
Such transform and their variants (on bounded domains and their boundaries) have attracted the interest of many authors, and play a crucial role in developing theory of holomorphic functions in complex analysis, Dirichlet problem in potential theory,  and in proving interpolation theorems, as well as in providing simple proof of Corona theorem.
The operator in \eqref{CT} is a bounded, compact integral operator on $L^2_g(\C)$.
 Moreover,  it belongs  to the $k$-Schatten class for every $k>2$.  
The spectral properties of $\Cc$ have been investigated by M.R. Dostani\'c  in \cite{Dostanic2000} and A. Intissar and A. Intissar in \cite{In}. 
In \cite{Dostanic2000}, Dostani\'c gave 
the exact asymptotic behavior of the singular values of $\Cc$ (the eigenvalues of $|\Cc|:= (\Cc^*\Cc)^{1/2}$), and of its orthogonal projection $P \Cc$ onto the classical Bargmann space $\mathcal{F}^2(\C)$ of entire functions in $L^2_g(\C)$.  
The polyanalytic setting is studied in \cite{In}, and make appeal to the  Hilbertian orthogonal decomposition of $L^2_g(\C)$
 in terms of the generalized Bargmann spaces $\mathcal{F}^2_n(\C)= \ker(\Delta - n)$, realized as the $n$-th $L^2$--eigenspace of the Landau operator $\Delta = \partial_z\partial_{\bz} - \bz \partial_{\bz}$. 
 The quaernionic analogs of $\mathcal{F}^2(\C)$ and $\mathcal{F}^2_n(\C)$ are introduced and studied in \cite{AlpayColomboSabadiniSalomon2014} and \cite{BenElhGh2019}, respectively.

Motivated by the aforementioned works, we propose to study the Cauchy singular transform in the quaternionic  setting, which involves additional technical problems for the lack of commutativity. As far as we know this has not been studied in the literature, despite that  
the Cauchy integral formula and the S-functional calculus are well known on the slice regular functions.     
 In the present paper, we deal with the quaternionic analog of the $\Cc$
 associated with the kernel function in \eqref{CkH}. To this purpose, we consider  the right Hilbert space 
 $\LgH := L^{2}_{\Hq}(\Hq;d\mu)$
  of quaternionic-valued square integrable functions on  $\Hq$ with respect to the Gaussian density $d\mu=e^{-|q|^2}d\lambda$, where $d\lambda$ is the standard Lebesgue measure on $\Hq\simeq \R^4$. The associated norm denoted by $\norm{\cdot}_{\Hq}$ is the one induced from the scalar product  
 \begin{equation}\label{spfgHq}
 	\scal{f,g}_{\Hq} = \int_{\Hq}\overline{f(q)} g (q) d\mu(q).
 \end{equation}
The weighted quaternionic Cauchy transform (WQCT), with respect to the measure $\mu$, is defined as a singular integral operator on $\LgH$ given by
\begin{equation}\label{wCauchyH}
\Ch f(q):=\dfrac{1}{\pi}\int_{\Hq}  \mathcal{N}(q,p) f(p)    d\mu(p)
\end{equation}
for $q\in \Hq$, where $\mathcal{N}$
 denotes the Cauchy kernel function given by
 \begin{eqnarray}\label{CauchyKernel} \mathcal{N}(p,q) :=  \left( (p-\bq)^{-1} p ((p-\bq)) - q \right)^{-1} . 
 \end{eqnarray}  
 The transform in \eqref{wCauchyH} is in fact a variant form of the one considered in \cite{ElhamyaniPhThesis2017}. Here, we have taking into account the linearity of such transform (when acting on right vector spaces) and the left slice regularity of the images when $\Ch$ acts on left slice regular functions.
 

Our main purpose is to investigate some spectral properties of the WQCT $\Ch $. 
 Mainly, we show that $\Ch$ is a bounded and compact operator in $\LgH$, belonging to the $k$-Schatten class for every $k>2$. We also give its explicit action on the so-called $n$-th S-polyregular Bargmann space of second kind $\mathcal{F}_{n,slice}^{2}$,  realized as the space of $L^2$-eigenfunctions of a slice differential operator \cite{BenElhGh2019}. This was possible using the quaternionic It\^o--Hermite polynomials that constitute a complete orthogonal system in $\LgH$. Moreover, we study the operator $P_n\Ch $ defined as the $n$-th Bergman projection of $\Ch$, $P_n$ being the $n$-Bergman projector onto $\mathcal{F}_{n,slice}^{2}$ given by 
	\begin{align} \label{OrthProjeq}
	P_nf(q) &= \dfrac{1}{\pi} \int_{\Hq}   e_{*}^{[\bq,p]} {\,\, {{\star}_{sp}^q}\,\, }  L_{n}(|q - p|_{{{\star}_{sp}^q}}^2 )    f(p) d\mu(p) ,
	\end{align} 
where $L_{n}$ is the $n$-th Lageurre polynomial and $\star$ denotes the star product for poluregular functions (these quantities are specified in the next section).
We explicit its Schwartz kernel function, we identify its range and determinate its singular values, i.e., the nonzero eigenvalues of $|P_{n}\Ch  |:=((P_{n}\Ch  )^{\ast}P_{n}\Ch)^{1/2}$.
Some of the obtained results generalize those elaborated in \cite{Dostanic2000,In,Gh2020} for the weighted complex Cauchy transform $\Cc$ on $L^2_g(\C)$.

 This paper is organized as follows.
 In Section 2, we review from \cite{BenElhGh2019,ElHamyani2018}, the basic tools related to the structure of the $n$-th S-polyregular Bargmann space. 
 In Section 3, we study the basic properties of the weighted quaternionic Cauchy transform $\Ch $, including boundedness, compactness, membership in $k$-Schatten class and its explicit action on the quaternionic It\^o--Hermite polynomials. Description of its range is also discussed. In Section 4, we study its  $n$-Bergman projection $P_{n}\Ch $ on $\mathcal{F}_{n,slice}^{2}$ and {the whole $\LgH$}, and we obtain
 the singular values as well as their exact asymptotic behavior.
 
 \section{PRELIMINARIES} 
  Let $\Hq$ denotes the real skew algebra of quaternions defined as the $4$-component extended complex numbers 
 $q=x_0+x_1\mathbf{i}+x_2\mathbf{j}+x_3\mathbf{k} \in{\Hq}$, where  $x_0,x_1,x_2,x_3\in{\mathbb{R}}$, endowed with the Hamiltonian computation rules so that the imaginary units $\mathbf{i}$, $\mathbf{j}$ and $\mathbf{k}$ satisfy 
 $\mathbf{i}^2=\mathbf{j}^2=\mathbf{k}^2=\mathbf{i}\mathbf{j}\mathbf{k}=-1$;
 $\mathbf{k}\mathbf{i}=-\mathbf{i}\mathbf{k}=\mathbf{j}.$
 The algebraic conjugate and the modulus are defined by 
 $x_0-x_1i-x_2j-x_3k$ and $|q|=\sqrt{q\overline{q}}$, respectively. 
 If $\mathbb{S}$ denotes the set of imaginary units $I^2=-1$,
 identified with the unit sphere $S^2$ in $\Im \Hq=\R \mathbf{i}+\R\mathbf{j}+\R\mathbf{k}$, then we can rewrite any $q\in \Hq$ as $ q=re^{I_q\theta}$ (polar representation) or as 
 $q=x+I_q y$ (slice representation). Here  $r=|q| \geq 0$, $I_q\in \mathbb{S}$, $\theta \in [0,2\pi[$ and $x,y\in \R$. The last representation gives rise to the notion of slice plane 
 $\C_{I} := \mathbb{R}+\mathbb{R}I$, so that 
 $$  \cap_{I\in \mathbb{S}} \C_{I}=\R \quad 
 \mbox{and} \quad  \Hq= \bigcup_{I\in \mathbb{S}} \C_{I}=  \bigcup_{I\in \wSq} \C_{I},$$
 where $\wSq$ denotes the hemisphere of purely imaginary quaternions defined as
 $$\wSq:= \{ I= \cos(\phi) \mathbf{i} +
 \sin(\phi)\cos(\psi) \mathbf{j} + \sin(\phi)\sin(\psi)  \mathbf{k}; \, \phi,\psi\in(0,\pi)
 \},$$
 and endowed with surface area measure $d\sigma(I) = \sin(\phi) d\phi d\psi$.  Thereafter,  $\As$ will denote the area of $\wSq$.
 
 In accordance with the slice representation of quaternionic numbers, Gentili and Struppa were able in 2007 to develop the theory of slice regular functions on specific domains of
 $\Hq$ (see \cite{GentiliStruppa07}). The introduced theory is similar to the one of holomorphic functions.
 A given quaternionic valued real differentiable function $f$, on $\Hq\equiv \R^4$, 
 is said to be left slice  regular, and we write $f \in \mathcal{SR}(\Hq)$, if its 
  restriction  $f|_{\C_I}$ to any slice $\C_{I}$; $I\in \Sq=\{q\in{\Hq};q^2=-1\}$, 
satisfies the Cauchy--Riemann equation $\overline{\partial_I} f=0$, where $\overline{\partial_I} f$ stands for the conjugate of the left slice derivative,  
 \begin{align}\label{dbar}
 ( \overline{\partial_I} f)(x+Iy) := 
 \dfrac{1}{2}\left(\frac{\partial }{\partial x}+I\frac{\partial }{\partial y}\right)f|_{\C_I}(x+yI).
 \end{align}
 One of the main tool in this theory is the so-called Splitting lemma stating that  for every slice regular function $f \in  \mathcal{SR} (\mathbb{H})$, and mutually perpendicular $I$ and $J$  in $\mathbb{S}^{}$, there are two holomorphic functions $F, G:\mathbb{H}\cap \mathbb{C}_{I} \longrightarrow \mathbb{C}_{I}$ such that for any $z = x + yI,$ it is
 $f_{I}(z) = F(z) + G(z)J$.
 Another, essentials result is the presentation formula asserting that for every $f \in  \mathcal{SR} (\mathbb{H}),$ we have the following equality,
 $$f(x + yI) =\frac{1}{2}(1-IJ)f(x + yJ)+\frac{1}{2}(1+IJ)f(x - yJ).$$ 
 The first functional space we are interested in is the slice quaternionic Bargmann-Fock space defined in \cite{AlpayColomboSabadiniSalomon2014} as   
  \begin{equation}\label{SliceHyperBarg}
   \mathcal{F}_{slice}^{2} := \left\{f= \sum_{n=0} q^n c_n; \,\, c_n \in \Hq; \quad \sum_{n=0}^{+\infty} n!|c_n|^{2}<+\infty \right\}.
    \end{equation}
 It can be realized as closed subspace of slice regular functions $f \in \mathcal{SR}(\Hq)$ in the Hilbert space $\LgS := L^{2}_{\Hq}(\C_I,d\mu_I)$
 endowed with the scalar product 
 \begin{equation}\label{spfg}
 \scal{\varphi,\psi}_{\C_I} = \int_{\C_I}\overline{\varphi(x+Iy)} \psi(x+Iy)  d\mu_I(q),
 \end{equation}
where $d\mu_I(x+Iy) = e^{-x^2-y^2}dxdy$. 
 In \cite{BenElhGh2019}, the space $\mathcal{F}_{slice}^{2}$ has been realized as a specific $L^2$-eigenspace of the semi-elliptic (slice) second-order differential operator
 \begin{equation}\label{Laplaceian}
 \Box_{q}= - \partial_s \overline{\partial_s} + \overline{q} \overline{\partial_s},
 \end{equation}
 where for $q=x+I_{q} y \in \Hq$, we have 
 \begin{equation}\label{120}
 \overline{\partial_s} f(q)=\left\{
 \begin{array}{ll}
 \overline{\partial_{I_q}} f 
 (x+I_q y)
 , & \hbox{if }  y\ne 0,  \\
 \dfrac{df}{dx}(x), & \hbox{if }  y=0.
 \end{array}
 \right. 
 \end{equation}
 
 The generalization to the $n$-polyregular functions $\srn(\Hq)$, i.e., those functions whose $(n+1)$-th slice derivative $\overline{\partial_I}^{n+1}$ vanishes identically on $\Hq$ for every $I\in \Sq$, is performed by considering  the eigenvalue problem  $ \Box_{q}f=n f$ for $f$ in $\LgH$. 
 Here we focus our attention on the $n$-th S-polyregular space of second kind, denoted $\mathcal{F}_{n,slice}^{2}$.
More explicitly, the space $ \mathcal{F}_{n,slice}^{2}$ is characterized as the right $\Hq$-vector space of all convergent series (on the whole $\Hq$) of the form
$$ f(q) = \sum_{m=0}^{+\infty} H^Q_{m,n}(q,\bq)\alpha_{m}$$
belonging to $\LgS $. The occurring  quaternionic coefficients $\alpha_{m}$ are independent of $q$ and $I$, and satisfy the growth condition (clearly independent of $I$)
$$ \sum_{m=0}^{+\infty} m!|\alpha_{m}|^2 <+\infty.$$ 
In the above formula,  $H^Q_{m,n}$ stands for the $(m,n)$-th quaternionic It\^o--Hermite polynomial \cite{ElHamyani2018}, explicitly given by
 \begin{align}\label{chp}
 H^Q_{m,n}(q,\overline{q})
 =   \sum_{\ell=0}^{m\wedge n}(-1)^{\ell} \ell! \binom{m}{\ell}\binom{n}{\ell} q^{m-\ell}\overline{q}^{n-\ell}. 
 \end{align} 
Here and elsewhere after $m\wedge n=\min(m, n)$ and $m\vee n=\max(m, n)$ for given nonnegative numbers.
 The system 
 $$ \varphi^Q_{m,n}(q,\overline{q}):= (\pi m!n!\As)^{-1/2} H^Q_{m,n}(q,\bq )$$
  with varying $m,n=0,1,2,\cdots,$ is orthonormal in  
 $\LgH$. For fixed $n$, they form a complete orthonormal system for $\mathcal{F}_{n,slice}^{2}$, endowed with the norm induced from \eqref{spfg}.
 To authorize $H^Q_{m,n}$ taking  negative index, we have to use its hypergeometric representation 
 \begin{align}
 H^Q_{m,n}(q,\bq) &=c_{m,n}  
  \frac{q^m\overline{q}^n}{|q|^{2(m\wedge n)}}  {_1F_1}\left( \begin{array}{c} -(m\wedge n) \\ |m-n|+1 \end{array}\bigg | |q|^{2}  \right)  \label{16negative}
 \end{align}
 with 
 	\begin{equation} \label{constcmn}
 c_{m,n}:=  \frac{(-1)^{m\wedge n} (m\vee n)!}{|m-n|!}.
 \end{equation}
 Here, ${_1F_1}$ denotes the classical Kummer's function  \cite[Ch. 6, p. 262]{MagnusOberhettingerSoni1966}
 \begin{align}\label{NewHypergeometricFct}
 {_1F_1}\left(  \begin{array}{c} a  \\ c \end{array} \bigg | x \right)
 =\frac{\Gamma(c)}{\Gamma(a)}\sum_{j=0}^\infty \frac{\Gamma(a+j)}{\Gamma(c+j)} \frac{x^j}{j!} 
 \end{align}
for given $a,c,x\in \R$  with  $c\ne 0,-1,-2,\cdots$.  
Thus, we write   
\begin{equation}\label{16moins1}
H^Q_{-1,n}(q,\bq) :=-\dfrac{\bq^{n+1}}{n+1}    {_1F_1}\left( \begin{array}{c} 1 \\ n+2 \end{array}\bigg | |q|^{2}  \right) .
\end{equation}
The sequence of considered spaces $\mathcal{F}_{n,slice}^{2}$ form an orthogonal slice Hilbertian decomposition for $\LgS $ 
 \cite{BenElhGh2019},
 \begin{align*}
\LgS =\bigoplus_{n= 0}^{+\infty} \mathcal{F}_{n,slice}^{2} ,
\end{align*}
in the sense that for every $\varphi\in \LgS$ there exists a sequence of quaternions  $\alpha_{m,n}^I$, depending in $I\in\Sq$, and satisfying the growth condition 
	\begin{align} \label{GrCint}
	\pi \sum_{m=0}^\infty \sum_{n=0}^\infty m!n! |\alpha_{m,n}^I|^2  < +\infty 
	\end{align}  
such that
$$ \varphi(q) = \sum_{m=0}^\infty \varphi^I_n(q) 
$$
with 
\begin{align} \label{expslice}
\varphi^I_n(x+Iy) = \sum_{m=0}^\infty H^Q_{m,n}(x+Iy, x-Iy) \alpha_{m,n}^I .
\end{align} 
The functions $\varphi^I_n$ can be extended to the whole $\Hq$ leading to a S-polyregular function of order $n-1$ by considering 
\begin{align} \label{expsliceExt}
\widetilde{\varphi}^I_n(x+Jy) = \sum_{m=0}^\infty H^Q_{m,n}(x+Jy, x-Jy) \alpha_{m,n}^I 
\end{align} 
for any $J\in \Sq$. Clearly, we have $\widetilde{\varphi}^I_n \in \mathcal{F}_{n,slice}^{2}$.
The situation is different when considering functions $f$ on $\Hq$ which admits a slice decomposition in terms of the quaternionic It\^o--Hermite polynomials, i.e. for each $I\in \Sq$, we have 
$$ f|_{\C_I}(q) = \sum_{m=0}^\infty f^I_n(q)  = \sum_{m=0}^\infty\sum_{n=0}^\infty H^Q_{m,n}(q,\bq) \alpha_{m,n}^I .$$
In this case the $f^I_n$ are S-polyregular but do not belong necessary  to $\mathcal{F}_{n,slice}^{2}$.
However,   
$ f^I_n \in \LgH$ (resp. $f\in \LgH$) if and only if the coefficients 
satisfy 
\begin{align} \label{GrCintAv}
\sum_{m=0}^{+\infty}  m!n!\int_{\wSq}|\alpha_{m,n}^I|^2 <+\infty 
\quad (resp. \, \sum_{m=0}^{+\infty} \sum_{n=0}^{+\infty} m!n!\int_{\wSq}|\alpha_{m,n}^I|^2 <+\infty).
\end{align}
This holds true when for example $\alpha_{m,n}^I$ are constant with respect to $I$. 
This conditions is fulfilled if we restrict ourself to the space $\Hnew$ of all quaternionic valued functions on $\Hq$, $f : \Hq \longrightarrow \Hq$, whose restriction
$f|_{\C_I}$, belongs to $L^{2}_{\Hq}(\C_I,d\mu_I)$ for any $I\in\wSq$
and such that   

\ \begin{equation}\label{newnorm}
\norm{ f} _{\infty}:=\sup_{I\in \wSq} \norm{f|_{\mathbb{C}_{I}}}_{\mathbb{C}_{I}} 
\end{equation}
$$ $$
is finite. 
	We claim that $\norm{\cdot}_{\infty}$ defines a norm in $\Hnew$, and that $(\Hnew, \norm{\cdot}_{\infty})$ is a  complete   space which is contained in $\LgH$ with  
	$ \norm{f}_{\Hq} \leq \As \norm{f}_{\infty}.$

We conclude this section by noticing that $ \mathcal{F}_{n,slice}^{2} \subset \Hnew$, and that its elements satisfy the estimate	\begin{equation}\label{ineq}
|f(q)|  \leq  \frac{e^{ \frac{|q|^2}{2}}}{\sqrt{\pi}}  \norm{f}_{\C_I}
\end{equation}
for any $I\in \Sq$. 
 This is to say, the evaluation map $\delta_{q}f=f(q)$, for every fixed $q\in \Hq$, is a continuous linear form on  $\mathcal{F}_{n,slice}^{2}$, and therefore $\mathcal{F}_{n,slice}^{2}$ is a reproducing kernel Hilbert space. Thus, there exists some function $\mathcal{K}_{n}$ on $\Hq\times \Hq$ such that $p \longmapsto \mathcal{K}_{n}(q,p)
= \overline{\mathcal{K}_{n}(p,q)}$ belongs to $\mathcal{F}_{n,slice}^{2}$ and satisfies the reproducing property 
$f(q) = \scal{\mathcal{K}_{n}(\cdot,q),f}_{\C_I}$
for every $f\in \mathcal{F}_{n,slice}^{2}$. This 
 kernel function can be expanded as 
 \begin{align}\label{ExpansionKer1}
 \mathcal{K}_{n}(q,p)&= \sum_{m=0}^\infty \frac{H^Q_{m,n}(q,\bq) \overline{H^Q_{m,n}(p,\bp)}}{\pi m!n!} .
 \end{align}
Its closed expression is given by \cite[Theorem 3.6]{BenElhGh2019} 
\begin{align}\label{ExpRepKer1}
\mathcal{K}_{n}(q,p)&= \dfrac{e_{*}^{[\overline{q},p]}}{\pi} {\, {\star}^q_{sp}\,} L 
(|q - p|_{{\star}^q_{sp}}^2 ),
\end{align}
where $|q - p|_{{\star}^q_{sp}}^2$  denotes the  S-regularization of the modulus function $q \longmapsto (p-q)(\bp-\bq)$ with respect to  the left star product $\star^q_{sp}$  for S-polyregular functions in $q$, $L$ the classical Laguerre polynomials, and $$e_{*}^{[a,b]}= \sum_{k=0}^\infty \frac{a ^k b^k}{k!}.$$ 

Moreover, the integral operator 
	\begin{align} \label{OrthProjslice}
	P_nf (q) &  = \int_{\C_I} \mathcal{K}_{n} (p,q) f(p) d\mu_I(p)   
	\end{align}
 is the orthogonal projection of $\LgS$ on $\mathcal{F}_{n,slice}^{2}$. 
The problem of extending $P_n$ to a larger Hilbert space can be handled by averaging the coefficients over the hemisphere $\wSq$. 
Thus, for each nonnegative integer $n$, we define the transform  $P_n$ on $\Hnew$ by the series
\begin{align} \label{OrthProjeqseriesint}
P_nf (q) &= \frac{1}{\As} \sum_{m=0}^\infty 
H^Q_{m,n} (q,\bq)  \int_{\wSq} \alpha_{m,n}^I d\sigma(I).
\end{align}
which is absolutely and uniformly  convergent on compact sets of $\Hq$ thanks to \eqref{GrCintAv}. It belongs to $ \mathcal{F}_{n,slice}^{2}$ and one recovers \eqref{OrthProjslice} when $f\in \LgS$. This justifies the use of the same notation $P_n$ for both projectors.   
Moreover, using the fact  
$$  \scal{ H^Q_{m,n} , f }_{\Hq} =  	\pi m!n! \int_{\wSq} \alpha_{m,n}^I d\sigma(I) $$
and the orthogonality of $H^Q_{m,n}$ in $ \LgH $, we can rewrite $P_n$ in \eqref{OrthProjeqseriesint}  as 
\begin{align} \label{OrthProjeqseriesint2}
P_nf (q) = \frac{1}{\As}
\scal{ \mathcal{K}_{n}(\cdot,q) , f}_{\Hq},
\end{align}   
Thus, we have proved the following 

\begin{lemma} 
	The transform $P_n$ in \eqref{OrthProjeqseriesint2} defines an orthogonal projection of 
	$\LgH$ onto the S-polyregular Bargmann space $\mathcal{F}_{n,slice}^{2}$. Its integral representation is given by \eqref{OrthProjeqseriesint2},
	\begin{eqnarray} \label{OrthProjeq}
P_nf(q) = \dfrac{1}{\pi \As} \int_{\Hq}   
L_{n}
(|q - p|_{ \star^q_{sp}}^2 ) 
{ \,\star^q_{sp} \, }
e_{*}^{[\bq,p]} f(p) d\mu(p) .
\end{eqnarray} 
\end{lemma}	

\begin{remark} 
	For $n=0$, $P_0$ is the quaternionic Bergman projector on the slice hyperholomorphic Bargmann space in \eqref{SliceHyperBarg}.
\end{remark}

\begin{definition} 
	  We call the transform $P_n$ the $n$-th Bergman projection. 
\end{definition}

 \section{WEIGHTED QUATERNIONIC CAUCHY SINGULAR INTEGRAL}
 
In order to develop the quaternionic analogue of the Cauchy transform, one can consider the kernel function $p\longmapsto (q-p)^{-1}$. However, this function is  in general neither left nor right slice regular, unless $q\in\R$. Instead, one has to consider the one given by the power series representation  
\begin{eqnarray}\label{CauchyKernelexpansion} \mathcal{N}(q,p) = \left\{\begin{array}{ll} 
\displaystyle -\sum\limits_{\ell= 0}^{+\infty}  q^{\ell}p^{-1-\ell}; & 
p\in \overline{B_{|q|}}^c \\
\displaystyle\sum\limits_{\ell= 0}^{+\infty} q^{-1-\ell}p^{\ell} ; & p\in B_{|q|} .
\end{array} \right. 
\end{eqnarray} 
where for given $q\in \Hq$, we have denoted by $B_{|q|}:=\{p\in\Hq; \, |p| <|q|\} $ the ball of radius $|q|$ in $\Hq\equiv \R^4$ and centered at the origin  and by $\overline{B_{|q|}}^c:= \{p\in\Hq; \, |p| >|q|\}$ the complement of its closure.
It  turns out to be an appropriate combination of the $p$-left and the $q$-right non-commutative Cauchy kernel series in the terminology of Colombo and his collaborators (see e.g.  \cite{ColomboGentiliSabadini2010,Gantner2014}).
 More precisely, the closed expression is the one  given by  \eqref{CauchyKernel} or equivalently  \cite[Theorem 2.10]{ColomboGentiliSabadini2010}
$$ \mathcal{N}(q,p)  = \left(q^{2}-2q\Re(p)+|p|^{2}\right)^{-1}(q-\bp)$$ 
for  $q\in \Hq$ and $p\in V_{q}  :=\left\{ p\in \Hq, pq-qp \ne 0\right\}$.
It should be mentioned here that the Cauchy kernel function in \eqref{CauchyKernel} as a function of $p$, can be identified as the regular inverse $(q - p)^{-\star}$ in the slice regular sense  (see \cite{GentiliStoppatoStruppa2013,GentiliStoppato2008}, which is the unique left slice regular extension of $\C_{q} \ni p \longmapsto (q-p)^{-1}$.
 The next uniform estimate is needed for the forthcoming investigation. 
 
 \begin{lemma}
 	For every $0 \leqslant r < 2$, we have	\begin{align}\label{EstimateKernel} \sup _{q\in \Hq }\int _{\Hq}  | \mathcal{N} (q,p) |^r d\mu(p)\leqslant \frac{ \As}{\pi^{r-1}} \Gamma\left( 1-\frac{r}{2}\right),
 		\end{align}
 		where $\Gamma(\cdot)$ is the Euler gamma-function.
 \end{lemma}

 \begin{proof} By observing that the modulus
 	$|\mathcal{N}(q, p)|$ is independent of $I_{p}$ and $I_{q}$ for $p=x+I_p y$ and $q=x'+I_qy'$ in $\Hq$, we obtain
 	\begin{align*}
 	\int _{\Hq}  | \mathcal{N} (q,p) |^{r} d\mu(p)
 	&=\int _{I\in\wSq}\int _{\C_I}  | \mathcal{N} (q,p) |^{r} d\mu_{I}(p)d\sigma(I)
 	\\& = \int _{I\in\wSq}\int _{\mathbb{C}_{I_{q}}}  | \mathcal{N} (q,p) |^{r} d\mu_{I_q}(p)d\sigma(I)
 		\\&=
 	\As \int _{\mathbb{C}_{I_{q}}}  | \mathcal{N} (q,p) |^{r} d\mu_{I_q}(p).
 	\end{align*}
 	But since for $p\in \C_{I_q}$, we have $\mathcal{N} (q,p) = (q-p)^{-1}$,  we can conclude for \eqref{EstimateKernel} by appealing to Lemma 3.2 in \cite{In}.
 \end{proof}

 \begin{corollary}
	For the specific case of $r=1$, we get 
	\begin{align}\label{EtimateKernelspecific} \sup _{q\in \Hq }\int _{\Hq}  | \mathcal{N} (q,p) | d\mu(p)\leqslant \sqrt{\pi}  \As .
	\end{align}
\end{corollary}
 
 Associated to the Schwartz kernel $\mathcal{N}(p,q)$, there are several ways to extend the definition of weighted Cauchy transform in \eqref{CT} to acts on the right quaternionic Hilbert space $\LgH$. Notice for instance that  
  $\mathcal{C}^{I}_{\Hq}f(q)$, as well as $\mathcal{C}^{S}_{\Hq}f(q)$ below, are defined by means of slices, i.e., by restricting the occurred integral to a fixed slice. For arbitrary $q\in \Hq$ and given fixed $I\in \Sq$ (independently of $q$), we set
  \begin{equation}\label{18I}
 \mathcal{C}^I_{\Hq}f(q):=\dfrac{1}{\pi}\int_{\C_{I}} \mathcal{N}(q,p) f|_{\C_I}(p) d\mu_I(p).
 \end{equation}

 \begin{remark}\label{rem}
 	Another kind of quaternionic weighted Cauchy transform, extending \eqref{CT}, involves dynamic slice depending in $q$, to wit 
 	\begin{equation}\label{18SS}
 		\displaystyle
 		\mathcal{C}^S_{\Hq}f(q):=\dfrac{1}{\pi} 
 		\left\{ \begin{array}{ll}
 			\displaystyle \int_{\C_{I_q}} \mathcal{N}(q,p) f(p)  d\mu_I(p); & \,\, q\in \Hq\setminus \R, \\
 			\displaystyle\int_{\R} f(t) \mathcal{N}(t,x)e^{-t^{2}}dt;& \,\, q=x\in \R.
 		\end{array} \right.
 	\end{equation}
 The spectral properties of this transform is quite similar to the one considered in the complex setting \cite{Dostanic2000,In}, and the restriction of $ \mathcal{C}^{\mathbf{i}}_{\Hq}f|_{\C}$ to the complex plane $\C=\C_{\mathbf{i}}$ with $f$ being complex valued function on $\Hq$ reduces also to \eqref{CT}.  
\end{remark}

\begin{remark}\label{rem2} 
The position of the kernel function $\mathcal{N}(q,p)$ in the left hand-side of the test function $f$ must be taken into account in order to get the linearity of the constructed integral transform when acting on right vector spaces. Otherwise, we have to regularize the product  $f(p) \mathcal{N}(p,q)$ in the integrand using the $\star$-product instead of the standard dot product. 
 This follows a general scheme for constructing linear quaternionic analogs of the classical integral transforms and to overcome technical difficulties for the lack of commutativity.
\end{remark}

According to the last remark, we can suggest 

\begin{equation}\label{18Cqta}
\Ch^{\star} f(q):=\dfrac{1}{\pi}\int_{\Hq} f(p){\, {\star}^{q,p}_{sp} \,} \mathcal{N}(q,p) d\mu(p),
\end{equation}
where the ${\star}^{q,p}_{sp}$-product is defined for given 
$\displaystyle f(q,p) = \sum_{m,n=0}^\infty q^mp^n a_{m,n} $ and
 $\displaystyle g(q,p) = \sum_{m,n=0}^\infty q^mp^n b_{m,n}$
by 
$$ (f {\star}^{q,p}_{sp} g) (q,p)=   \sum_{\ell,s=0}^\infty q^\ell p^s \left( \sum_{m+n=\ell}^\infty 
\sum_{j+k=s}^\infty  a_{m,n} b_{j,k}\right)  
= (f {\star}^{q}_{sp}{\star}^{p}_{sp} g ) (q,p).$$
One can also consider the singular integral operator defined by
\begin{equation}\label{18Cqtb}
\Ch f(q):=\dfrac{1}{\pi}\int_{\Hq} \mathcal{N}(q,p) f(p)  d\mu(p) . 
\end{equation}
The following result shows that the Cauchy transform $\Ch$ can be seen as the averaging operator of the $I$-slice Cauchy transform $\Ch^{I}$ over the hemisphere $\wSq$ of purely imaginary quaternions.

\begin{lemma} \label{lemeq}
	We have 
	$$ \Ch f(q) 
	= \int_{\wSq} \Ch^{I} f(q) d\sigma(I).
	$$ 
\end{lemma}

\begin{proof} 
	The identity is immediate by definition of the integral on $\Hq$.  
\end{proof}

Accordingly,  the uniform boundedness	of the operators  $\Ch^I$ restricted to $\mathcal{F}_{n,slice}^{2}$, when $I\in \wSq$, is a sufficient condition for the boundedness of $\Ch$ on $\mathcal{F}_{n,slice}^{2}$. In fact, by applying the Cauchy-Schwartz inequality, we get 	
$$ 
\norm{\Ch f}_{\Hq}^2= O   \left( \sup_{I\in \wSq}    \norm{\Ch^I}^2  \norm{f}_{\C_I}^2\right) .$$
Now, since $ \norm{f}_{\C_I}^2$ is independent of $I$ when $f\in \mathcal{F}_{n,slice}^{2}$, we get $\norm{f}_{\C_I}^2= O (\norm{f}_{\Hq}^2) $ and therefore
$$ 
\norm{\Ch f}_{\Hq}^2= O  \left( \sup_{I\in \Sq}    \norm{\Ch^I}^2  \right) \norm{f}_{\Hq}^2.$$
Clearly, this argument is also valid for $\Ch$ acting on $\Hnew$.
We note also that the boundedness of $\Ch$ can also be deduced by applying the Schur sufficiency condition  for boundedness of integral operator with Schwartz kernel. Below we  present a variant proof. To this purpose, we need to recall the Russo's Lemma \cite{russo}, for compactness and membership in $k$-Schatten class of integral operator with Schwartz kernel.

	\begin{lemma}[Russo's Lemma \cite{russo}] \label{lemRusso}
	An integral operator of the form 
	$$Tf(y) = \int _{X} \mathcal{S} (y,x)f(x)d\nu(x),$$
	for given finite measure $\nu$ on a given no empty set $X$, belongs to the $k$-Schatten class as an operator on $L^2(X,d\nu)$, if its kernel function satisfies 
	$$\int _{\Hq} \left( \int _{\Hq}|\mathcal{S} (q,p)|^{r}d\nu(p)\right)^{k/r} d\nu(q) \,\, 	\mbox{and} 	\,\, 
	\int _{\Hq} \left( \int _{\Hq}|\mathcal{S} (q,p)|^{r}d\nu(q)\right)^{k/r} d\nu(p) $$
	are finite for  $1 \leq r < 2$ which is the exponent conjugate  of $k$,  $1/k + 1/r = 1$.
\end{lemma}

\begin{theorem}
	The weighted Cauchy transform $\Ch$ is a bounded and compact linear operator from $\LgH$ to $\LgH$ with norm operator satisfying  $\norm{\Ch} \leqslant  {\As}/{\sqrt{\pi}} .$ 
	Moreover, for every $k > 2$ it belongs to the $k$-Schatten ideal of bounded operators $S_{k}(L^{2}_{\mu}(\Hq))$.
\end{theorem}

\begin{proof}
	The Cauchy-Schwartz inequality and Lemma \ref{EstimateKernel} imply 
	\begin{align*}
	|\mathcal{C}_{\Hq}(f)(q)|^{2}
	&\leqslant\frac{1}{\pi^{2}}\left(  \int_{\Hq}|\mathcal{N}(q, p)|d\mu(p)\right)  \left( \int_{\Hq}|\mathcal{N}(q, p)|| f(p)|^{2}d\mu(p)\right) \\&\leqslant \frac{\As}{\pi^{3/2}} \int_{\Hq}|\mathcal{N}(q, p)| | f(p)|^{2}d\mu(p) .
	\end{align*} 
Therefore, 
	\begin{align*}
	\norm{\Ch(f)}_{\Hq}^2
	&\leqslant \frac{\As}{\pi^{3/2}}\int _{\Hq}  \int_{\Hq}|\mathcal{N}(q, p)|| f(p)|^{2}d\mu(p) d\mu(q)\\
	&\leqslant \frac{\As}{\pi^{3/2}} \int _{\Hq}   |f(p)|^{2} \left( \int_{\Hq}|\mathcal{N}(q, p)|d\mu(q)\right) d\mu(p) \\&\leqslant \frac{\As}{\pi^{3/2}}  \left( \sup _{p\in \Hq } \int_{\Hq}|\mathcal{N}(q, p)|d\mu(q) \right) \norm{f}_{\Hq}^2.
	\end{align*}
	Finally, by Lemma \ref{EstimateKernel} again, we get
	$$
	\norm{\Ch(f)}_{\Hq}^2
\leqslant \frac{(\As)^2}{\pi }   \norm{f}_{\Hq}^2 .
	$$
	The compactness of $\Ch$ follows in a similar way as in \cite{In} for the complex Cauchy transform \eqref{CT}.
	Indeed,  using the estimation provided by Lemma \ref{EstimateKernel} and the fact that $d\mu$ is a finite  measure on $\Hq$, we get
	\begin{align*}
	\int _{\Hq} \left( \int _{\Hq}  | \mathcal{N} (q,p) |^{r} d\mu(p)\right) ^{k/r}d\mu(q)
	&\leqslant 
\frac{\pi (\As )^2}{2} \left( \frac{ \Gamma\left( 1-\frac{r}{2}  \right)}{\pi^{r-1}}\right)  ^{k/r} 
	\end{align*}
	for every $1 \leq r < 2$, where $k$ is its exponent conjugate $1/k + 1/r = 1$.
	Thus,  since  $| \mathcal{N}(q,p)|=| \mathcal{N}(p,q)|$, the two requirements in Russo's lemma (Lemma  \ref{lemRusso}) are satisfied.
\end{proof}

	\begin{remark} 		
		The compactness of $\Ch$ can also be studied by performing the operator $\Ch^*\Ch$ and proceed as in \cite{ArazyKhavinson1992}, where $\Ch^*$ denotes the adjoint of $\Ch$ defined as operator from $\LgH$ into $\LgH$  by 
		$$ \Ch^*(g)(p) =\frac{1}{\pi} \int_{\Hq} \overline{\mathcal{N}(q,p)} g(q) d\mu(q)= -\frac{1}{\pi} \int_{\Hq} \mathcal{N}(\bp,\bq) g(q) d\mu(q).$$
\end{remark}

The theorem just proved and the upcoming ones remain valid for the Cauchy transform  $\Ch^{\star}$ in \eqref{18Cqta}. This readily follows by observing that $\Ch^{\star} $   
coincides with $\Ch$ in \eqref{18Cqtb} (see  Corollary \ref{coreq} below). For the proof, we use the following result giving the explicit expression for $\Ch $ and $\Ch^{\star} $ on the generic elements $e_{m,n}(q,\overline{q}) := q^{m} \overline{q}^n$. 
 For exact statement, we let $\varepsilon_k$ stands for 
 $$ \varepsilon_k= \left\{ \begin{array}{ll} 
 1 & \mbox{if } \, k \geq 0, \\
 0 & \mbox{if } \, k<0.  \end{array} \right.   $$

 \begin{proposition}\label{propactionemn} 
 	For every nonnegative integers $m,n$ and $q\in\Hq$, we have $\Ch^{\star} (e_{m,n})    = \Ch (e_{m,n})$ and
 	\begin{align}\label{actiongeneric2}
 	\left[\Ch^{\star} (e_{m,n})\right](q) &=   \As\left( 
 	q^{m} H^Q_{-1,n}(q,\bq) e^{-|q|^2} - n!\varepsilon_{m-n-1} q^{m-n-1}\right) . 
 	\end{align}
 \end{proposition}
 
 \begin{proof}
 	We need only to prove the  identity \eqref{actiongeneric2}, since the computation of $\Ch (e_{m,n})$ can be handled in a similar way to give rise to the right hand-side in \eqref{actiongeneric2}.  
 	To this end, we keep notation of the ball of radius $B_{|q|}$, the complement of its closure $\overline{B_{|q|}}^c$ and the hemisphere $\wSq$ as above.  
 	Thus, making use of the expansion series \eqref{CauchyKernelexpansion} of the Cauchy kernel $\mathcal{N}$,
	we can rewrite \eqref{18Cqta} as 
 	\begin{align}
 	\left[\Ch^{\star} (e_{m,n})\right](q) 
  	& 	=\dfrac{1}{\pi}\left( \int_{B_{|q|}} p^m\bp^n  {\, {\star}^q {\star}^p \,} \mathcal{N}(q,p) d\mu(p) + \int_{\overline{B_{|q|}}^c} p^m\bp^n  {\, {\star}^q {\star}^p \,} \mathcal{N}(q,p) d\mu(p) \right)\nonumber \\
 	&= \dfrac{1}{\pi}\sum\limits_{\ell= 0}^{+\infty}\left( q^{-1-\ell} \int_{B_{|q|}} p^{m+\ell} \overline{p}^{n}d\mu(p)
 		- q^{\ell} \int_{\overline{B_{|q|}}^c} p^{m-1-\ell}\overline{p}^{n}   d\mu(p)\right) \nonumber
 	\\	&=\dfrac{1}{\pi}\sum\limits_{\ell= 0}^{+\infty} 
 	q^{-1-\ell}\int_{|q|}^{\infty}  \int_{ \wSq} \int_0^{2\pi} r^{m+\ell+n} e^{I(m+\ell-n)\theta} e^{-r^{2}} rdrd\sigma(I)d\theta  \nonumber\\ & \qquad - \dfrac{1}{\pi}\sum\limits_{\ell= 0}^{+\infty} q^{\ell} \int_0^{|q|}   \int_{ \wSq} \int_0^{2\pi}
 	r^{m-1-\ell+n} e^{I(m-1-\ell-n)\theta} e^{-r^{2}} rdrd\sigma(I) d\theta 
 	\nonumber \\
 	&= \As q^{m-n-1} \left( \varepsilon_{n-m}  \int_0^{|q|^2}
 	t^{n}  e^{-t} dt   - 
 	\varepsilon_{m-n-1}   \int_{|q|^2}^{\infty} t^{n}  e^{-t} dt   \right)  . 
 	\label{actiongeneric} 
 	\end{align}
 	Here we have used the polar coordinates $q=re^{I\theta}$ with  $r > 0$, $\theta \in [0,2\pi]$, $I$ in the hemisphere $\wSq$, and next the change of variable $t=r^2$. 
 	Above $dr$ and $d\theta$ are the Lebesgue measure on positive real line and unit circle, respectively. 
 	Therefore, keeping in mind the expression of $H^Q_{m-1,n}(q,\bq)$ in \eqref{16moins1}, 
 	we can check \eqref{actiongeneric2} by observing that 
 		\begin{align*}
 	\psi_n(|q|^2) &:= \int_0^{|q|^2} t^{n}  e^{-t} dt =  n! - \int_{|q|^2}^{\infty}
 	t^{n}  e^{-t}, 
 	\end{align*}
 	and using 	 $\varepsilon_{n-m}  + \varepsilon_{m-n-1}=1$,
 as well as
 	\begin{align}
 	\psi_n(|q|^2)  	&= \frac{|q|^{2(n+1)}}{n+1} {_1F_1}\left( \begin{array}{c} 1 \\ n+2 \end{array}\bigg | |q|^2 \right) e^{-|q|^2} . \label{factoriel}
 	\end{align}
 	The last equality follows thanks to the integral representation of confluent hypergeometric function \cite[p. 275]{MagnusOberhettingerSoni1966}
 	\begin{equation}\label{22}
 	{_1F_1}\left( \begin{array}{c} a \\ c \end{array}\bigg | x \right)=\frac{\Gamma(c)}{\Gamma(a)\Gamma(c-a)} e^z \int_{0}^{1}e^{-zt}t^{c-a-1} (1-t)^{a-1} dt
 	\end{equation}
 	with $a=1$ and $c=n+2$.	
\end{proof}

 \begin{remark}
 	The result of Proposition \eqref{propactionemn}
 	generalizes the one obtained in \cite{AnderssonHinakkenen1989} for the standard unweighted complex Cauchy transform on the monomials $z^m\bz^n$.
 \end{remark}

\begin{corollary} \label{coreq}
	We have 
	$\Ch^{\star}  = \Ch $.
\end{corollary}
  
  \begin{proof} 
    This readily follows form Proposition \ref{propactionemn} since the polynomials 
    $e_{m,n}(q) = q^m\bq^n$ are dense in $\LgH$.     
  \end{proof}

 In order to provide concrete description of basic properties of $\Ch $, we need to its explicit expression when acting on the quaternionic It\^o--Hermite polynomials $H^Q_{m,n}(q,\overline{q} )$. From now on, we normalize the area measure on the hemisphere $\wSq$ such that $\As=1$. 
 
\begin{theorem}\label{thmM1thm23} 
	The explicit action of $\Ch $ on the orthonormal basis $$\varphi^Q_{m,n}=(\pi m!n!)^{-1/2}H^Q_{m,n},$$ for every $m>0$, is given by 
		\begin{equation}\label{actionCHQ1}
	\left[\Ch  \varphi^Q_{m,n}\right](q)=  -\frac{e^{-|q|^{2}}}{\sqrt{m}} \varphi^Q_{m-1,n}(q,\bq),
	\end{equation}
while for $m=0$, we have
	$$ \left[\Ch  \varphi^Q_{0,n}\right](q) 
	=- \dfrac{ e^{-|q|^{2}} }{\sqrt{\pi n!}}   H^Q_{-1,n}(q,\bq).$$ 
\end{theorem}

\begin{proof}
	Notice first that the specific case of $m=0$ 
 can be checked easily by taking $m=0$ in Proposition \ref{propactionemn}. Indeed, 
\begin{align*}
	\Ch (H^Q_{0,n}) =   \Ch (e_{0,n})= -
	\frac{\overline{q}^{n+1}}{n+1}    {_1F_1}\left( \begin{array}{c} 1 \\ n+2 \end{array}\bigg | |q|^2 \right)
	e^{-|q|^2} 
	\end{align*}
	which is close to the hypergeometric function $H^Q_{-1,n}(q,\bq)$ in \eqref{16moins1}.
	
For $m>0$. we make use of \eqref {actiongeneric} and the linearity of $\Ch $ to have
 	\begin{align*} 
\Ch H_{m,n}(q)&=  
 q^{m-n-1} \left( \varepsilon_{n-m}  I_{m,n}(|q|^2) - \varepsilon_{m-n-1} J_{m,n}(|q|^2)\right) , 
\end{align*}
 with
	\begin{align*} 
I_{m,n}(|q|^2): 
=
\frac{ (-1)^{m} n!}{(n-m)!}    \int_0^{|q|^2}
t^{n-m}
{_1F_1}\left( \begin{array}{c} -m \\ n-m +1 \end{array}\bigg | t \right) e^{-t} dt
 \end{align*}
and
	\begin{align*}  
J_{m,n}(|q|^2): 
=   \int_{|q|^2}^{\infty} {_1F_1}\left( \begin{array}{c} -n \\ m-n +1 \end{array}\bigg | t \right) e^{-t} dt.  
\end{align*}
Direct	 computation
	 making appeal to indefinite integrals for the confluent hypergeometric 
	function ${_1F_1}$, see for instance \cite[p. 266]{MagnusOberhettingerSoni1966}, or equivalently the differentiation formulas in \cite[4, p. 73]{Brychkov2008} 
	with $a=1-m$,  $b=n-m+2$ and  $c=1=k$ for the evaluation of $I_{m,n}(|q|^2)$
 and  
	\cite[6, p. 73]{Brychkov2008} 
	with $a=-n$,  $b=|m-n|$ and $c=1=k$ for $J_{m,n}(|q|^2)$, 
 keeping in mind the hypergeometric representation of $H^Q_{m-1,n}(q,\bq)$ given through \eqref{16negative}, we arrive at \begin{align}\label{Casem0}
	I_{m,n}(|q|^2)  	= - q^{n-m+1} H^Q_{m-1,n}(q,\bq)  e^{-{|q|^2}} = -J_{m,n}(|q|^2).  
	\end{align}
	Therefore, the identity \eqref{actionCHQ1} follows since $\varepsilon_{n-m}  + \varepsilon_{m-n-1}=1$.
\end{proof}

\begin{remark}	
	Green formula was employed  in \cite{AnderssonHinakkenen1989,ArazyKhavinson1992,Dostanic1996} to get the explicit expression of the unweighted complex Cauchy transform on bounded domains when acting on specific elementary functions. See also \cite{In} for the weighted complex Cauchy transform on the complex It\^o--Hermite polynomials of total degree great than $1$.
	Adoption of similar approach for the quaternionic Cauchy transform requires an appropriate quaternionic analogue of this famous Green formula (which can be obtained starting from quaternionic version of Cauchy representation theorem \cite{ColomboGentiliSabadini2010,ColomboMongodi2019}). 
	Here, we have given a direct proof using the expansion series of the kernel function combined with the existing integral formulas for special functions.
\end{remark}

For every nonnegative integers $m,n$, we denote by $ \psi_{m,n}$  the Hermite functions defined by 
\begin{align}\label{psifct} \psi_{m,n}(p) := -e^{-|p|^2}H^Q_{m,n-1} (p,\bp) = \overline{\Ch (H^Q_{n,m})(p)}.
\end{align} 
Accordingly, for given nonnegative integer $j$, we perform the spaces
 $$E_{j}^+   =  \overline{span\{\psi_{n, n+j}; \, n=0,1,2, \cdots \}}^{\LgH} $$
and 
$$E_{j}^- = \overline{span\{ \psi_{n+j,n }; \, n=0,1,2, \cdots \}}^{\LgH}$$ 
as well as the spaces 
$E_\ell = E_{|\ell|}^+$ if $\ell \geq 0$ and $E_\ell = E_{|\ell|}^-$ when $\ell< 0$.

\begin{theorem}	The spaces 
	$ E_\ell$, for varying integer $\ell$,
	 form an orthogonal Hilbertian decomposition for the range of the weighted Cauchy transform $\Ch$ in  {$\LgH$}.
\end{theorem}   

\begin{proof}
	We begin by noticing that $\psi_{n,m} \in \LgH $. Indeed, we have
	$$ \norm{ \psi_{n,m}}_{\Hq}^2 = \norm{ e^{-|p|^2} H^Q_{m-1,n}}_{\Hq}^2 
	\leq \As\norm{  H^Q_{m-1,n}}^2_{\C_I} < \infty
		$$
	by means of (v) in \cite[Proposition 3.2]{In}.
	Next, by Theorem \ref{thmM1thm23}, the hypergeometric representation \eqref{16negative} of $H^Q_{m,n}$ and Fubini's theorem, we can show that 
	the orthogonality of the system $(\psi_{n,m})_{m,n}$ in $\LgH $ is equivalent to the nullity of the angular part in the integral giving $ \scal{\psi_{n,m},\psi_{k,j}}_{\Hq}$  given by 
	$$ A_{m,n,j,k}  
	:= \int_0^{2\pi} e^{I_q(n+j-m-k)\theta} d\theta = 2\pi \delta_{m-j,n-k} .$$
	Indeed, direct computation shows that  
	\begin{align}\label{scalarChCh}
	\scal{\psi_{n,m},\psi_{k,j}}_{\Hq} 
	& = \pi  c_{m-1,n} c_{j-1,k} 	 \delta_{m-j,n-k}  \int_0^{\infty}
	\frac{e^{-3t} t^{m+n-1}  R_{m,n,j,k}(t)  }{t^{\min(m-1,n)+\min(j-1,k)}}dt ,
	\end{align}
	where the constants $c_{m,n}$ are as in \eqref{constcmn} and $ R_{m,n,j,k}$ is 
	given by 
	\begin{equation} \label{ptodHyp}
	R_{m,n,j,k}(t) :=  {_1F_1}\left( \begin{array}{c} - (m-1)\wedge n \\ |m-1-n|+1\end{array}\bigg | t \right)
	{_1F_1}\left( \begin{array}{c} -(j-1)\wedge k \\ |j-1-k|+1\end{array}\bigg | t \right) .
	\end{equation}
	Thus, for $m-j=\ell \ne n-k =\ell'$, we have
	$	\scal{\psi_{n,m},\psi_{k,j}}_{\Hq} =0$. This proves in particular that the $E_\ell$; $\ell\in \Z$, form an orthogonal sequence in $\LgH$.  Moreover, we have  
 $$\bigoplus_{\ell\in\Z} E_\ell  
 = \Ch\left( \bigoplus_{n=0}^\infty \mathcal{F}_{n,slice}^{2} \right) .$$ 
\end{proof}

\begin{corollary}
	\label{thm23} 
	The functions  $\psi_{n,m}$,  for varying $m=0,1, \cdots,$ and fixed $n$, (resp. for varying $n=0,1, \cdots,$ and fixed $m$) constitute an orthogonal system in $\LgH$ whose square norm is given by 
	\begin{equation}\label{normCH}
\norm{ \psi_{n,m}}_{\Hq}^2 =   \frac{\pi}{3^{m+n}} \left\{\begin{array}{ll}   
\displaystyle \frac{4^{m-1 } ((n)!)^2  }{(n-m+1)!}
{_2F_1}\left( \begin{array}{c} 1-m  , 1-m   \\ n-m+2\end{array}\bigg | \dfrac1 4 \right); & m\leq n+1,\\\\
 \displaystyle   \frac{4^{ n } \left((m-1 )!\right)^2  }{(m-1-n)!}
{_2F_1}\left( \begin{array}{c} -  n  , - n  \\ m-n\end{array}\bigg | \dfrac1 4 \right); & m \geq n +1 .
\end{array} \right.  
\end{equation} 
\end{corollary}

\begin{proof} 
	The orthogonality in $\LgH$ of the systems 
	$(\psi_{n,m})_{m}$, for fixed $n$, and $(\psi^Q_{n,m})_{n}$, for fixed $m$, readily follows from \eqref{scalarChCh}. We need only to compute its norm. We have  
	\begin{align*}
	\norm{\psi_{n,m}}_{\Hq}^2 	&= \pi (c_{m-1,n})^2  \int_0^{\infty} t^{|m-1-n|}  	R_{m,n,m,n}(t) e^{-3t}dt.
	\end{align*}	
	Thus, we get \eqref{normCH} by appealing to the integral formula \cite[p. 293]{MagnusOberhettingerSoni1966}.
\end{proof}

\section{THE $k$-BERGMAN PROJECTION OF $\Ch$}
 
 Corollary \ref{thm23}  just proved 
  shows that the functions $\psi_{n,m}= \Ch H^Q_{m,n}$; $m;n=0,1,2,\cdots ,$ belong to $\LgH $, possessing the slice orthogonal decomposition in the sense described in the preliminaries section. 
 Accordingly, the determination of the component function of $\psi_{n,m}$ in $\mathcal{F}_{\ell,slice}^{2}$ requires the  consideration of the orthogonal projection in \eqref{OrthProjeqseriesint2}. 
The next assertion gives the closed integral representation of $P_k\Ch$, the $k$-Bergman projection of the weighted Cauchy singular integral $\Ch$, in terms of the special polynomials 
  \begin{align}\label{chpstar}
G^{Q,\star^q}_{m,n}(q - p,\bq-\bp)  :=  
\sum_{\ell=0}^{m\wedge n}(-1)^{\ell} \ell! \binom{m}{\ell}\binom{n}{\ell} (q - p)^{\star^q(m-\ell)} \star^q (\bq-\bp)^{\star^{\bq}(n-\ell)},
\end{align} 
obtained as the unique left slice polyregular extension in $q$ and right slice polyregular in $\bq$ to the whole $\Hq$ of the facts
$q \longmapsto G^{Q,\star^q}_{k-1,k}(q - p,\bq-\bp)  $ on $\C_p$, for fixed $p$.
It is also the unique right slice polyregular extension in $p$ and left slice polyregular in $\bp$ to the whole $\Hq$ of 
$p \longmapsto G^{Q,\star^q}_{k-1,k}(q - p,\bq-\bp)  $ on $\C_q$, for fixed $q$.

 \begin{theorem} 
	The integral transform $P_k\Ch : \LgH  \longrightarrow \mathcal{F}_{k,slice}^{2}$ is given by 
	\begin{equation}\label{PkC}
	P_k\Ch f(q) = \int_{\Hq}  \mathcal{R}_{k}(q,p) f(p) d\mu(p) 
	\end{equation}
	where $\mathcal{R}_{k}$ stands for 
		\begin{align}\label{PkCkernel}
 \mathcal{R}_{k}(q,p) := 
  \dfrac{1}{\pi k!} e^{-|p|^2} e_{*}^{[q,\bp]} {\,\, {{\star}_{sp}^q}\,\, } H^{\star}_{k-1,k}(q - p,\bq-\bp) . 
	\end{align} 	
\end{theorem}	

\begin{proof} 
By means of \eqref{OrthProjeqseriesint2}, the definition of $\Ch$ and the Fubini's theorem, 
\begin{align*} 
P_k\Ch f(q) &= \int_{\Hq}  \mathcal{K}_{k}(q,\xi)  \Ch f(\xi) d\mu(\xi)  
= \int_{\Hq}  \mathcal{R}_{k}(q,p)f(p) d\mu(p)
\end{align*}
for $q\in\Hq$, where the kernel function is given by 	
\begin{align} 
\mathcal{R}_{k}(q,p) & =\int_{\Hq}  \mathcal{K}_{k}(q,\xi) \mathcal{N}(\xi,p) d\mu(\xi) \nonumber\\
& = \sum_{m=0}^\infty \frac{H^Q_{m,k} (q,\bq) }{\pi m! k!} \int_{\Hq}    \overline{H^Q_{m,k} (\xi,\bxi)}   \mathcal{N}(\xi,p) d\mu(\xi) \nonumber\\
& = -\sum_{m=0}^\infty \frac{H^Q_{m,k} (q,\bq) }{\pi m! k!} \Ch H^Q_{k,m} (p)   \nonumber \\
& = \sum_{m=0}^\infty \frac{H^Q_{m,k} (q,\bq) }{\pi m! k!}  H^Q_{k-1,m} (p,\bp)  e^{-|p|^2}. \label{expRk}
\end{align}

Above we have made use of the expansion series of $\mathcal{K}_{k}$ as given by \eqref{ExpansionKer1} and the facts 
$\overline{\mathcal{N}(\xi,p) }= -\mathcal{N}(\bp,\bxi)$,
$\overline{H^Q_{k,m}(p,\bp)} = H^Q_{m,k}(p,\bp) = H^Q_{k,m}(\bp,p),$
 and $\overline{\Ch H^Q_{k,m} (\bp)} = \Ch H^Q_{k,m} (p)$.
  The last equality \eqref{expRk} follows from Theorem \ref{actionCHQ1}.
Now, the function 
\begin{align*} 
\widetilde{\mathcal{R}}_{k}(q,p) := e^{|p|^2} \mathcal{R}_{k}(q,p)
& = \sum_{m=0}^\infty \frac{H^Q_{m,k} (q,\bq) }{\pi m! k!}  H^Q_{k-1,m} (p,\bp)  
\end{align*}
is clearly left slice polyregular of order $k+1$ in $q$ and right slice polyregular of order $k$ in $p$, since it can be rewritten as
$$\widetilde{\mathcal{R}}_{k}(q,p) = \sum_{j=0}^{k} \bq^j f^{(p)}_j(q) = \sum_{j=0}^{k-1}  g^{(q)}_j(p) \bp^j$$
with $f^{(p)}_j$ (resp. $g^{(q)}_j$) being left (resp. right) slice regular functions in $q$ (resp. in $p$). Moreover, Proposition 3.6 in \cite{Gh13ITSF} shows that 
 the expression of 
$\widetilde{\mathcal{R}}_{k}(q,p) := e^{|p|^2} \mathcal{R}_{k}(q,p)$ reduces further to 
$$ E_{k}(q,p) = \frac{e^{ q\bp}}{\pi k!}  H_{k,k-1}(q-p,\bq-\bp)  $$
on $\C_{I_p}$ (resp. $\C_{I_q}$)  as function in $q$ (resp. $p$) for fixed $p$ (resp. $q$).
Finally, using Identity Principle for slice regular functions,  one can proceed as in \cite{BenElhGh2019} to prove that the function in the right hand-side of \eqref{PkCkernel} (in the $(q,p)$ variables) is the  
the unique extension  to left slice hyperholomorphic function in $q$  and right one in $\bq$ outside of $\C_p$, and right slice hyperholomorphic in $p$ and left slice hyperholomorphic in $\bp$ outside of $\C_q$.  Thus, we have 
$$ \widetilde{\mathcal{R}}_{k}(q,p) = \dfrac{e_{*}^{[q,\bp]}}{\pi k!}   {\,\, {{\star}_{sp}^q}\,\, } G^{Q,\star^q}_{m,n}(q - p,\bq-\bp)  .$$ 
This completes the proof.
\end{proof}

 \begin{proposition}   
For every nonnegative integers $k,m,n$, there exists certain nonzero real constant  $d_k^{m,n}$ such that
	\begin{align} \label{actionPCpsig}
	P_k\Ch  \psi_{m,n} (q) =     \frac{\varepsilon_{n+k-m} d_k^{m,n}  }{\pi (n+k-m)! k!}   H^Q_{n+k-m,k} (q,\bq) .
	\end{align}
	In particular, the systems  $P_k\Ch \left( \psi_{m,n}\right)$; $n=0,1,2, \cdots$, for fixed $m$, and   $P_k\Ch \left( \psi_{m,n}\right)$; $m=0,1,2, \cdots$, for fixed $n$ are orthogonal in $\LgH$. 
\end{proposition}	

\begin{proof} 
	By taking into account the expansion series of the kernel given by \eqref{expRk}, we can rewrite the integral in \eqref{PkC} in terms of $\psi_{j,k}$ in \eqref{psifct}, as
	\begin{align*}
	P_k\Ch f(q) 
	&= \sum_{j=0}^\infty  \frac{H^Q_{j,k} (q,\bq) }{\pi j! k!} 
	\int_{\Hq}  \overline{\psi_{k,j}(p)} f(p) d\mu(p) .
	\end{align*} 
	So that the only possible nonzero term corresponds to $j =n+k-m$ under the assumption that $n+k\geq m$. This follows by specifying $f= \psi_{m,n}$, 
	and using \eqref{scalarChCh}. Thus, we have  
	\begin{align*}
	P_k\Ch  \psi_{m,n} (q) =     \frac{\varepsilon_{n+k-m} }{\pi (n+k-m)! k!}  d_k^{m,n}  H^Q_{n+k-m,k} (q,\bq) .
	\end{align*}
	The constant $d_k^{m,n}$ is given by 
	$ d_k^{m,n}
	:=    \scal{ \psi_{k,n+k-m}, \psi_{m,n}  }_{\Hq}$ which is finite and real.
\end{proof}

\begin{remark} The polynomials $H^Q_{n,k}$ are eigenfunctions of the operator $P_k\Ch M_g \Ch$, where $M_g$ is the conjugation operator  $M_g(f) (p)= \overline{f(p)}$.
	 This clearly follows from the fact 
	\begin{align}\label{actionPCpsi}
	P_k\Ch \left( \psi_{k,n}\right) (q) 
	& =  \frac{\norm{ \Ch (H^Q_{k,n}) }^2_{\Hq} }{\pi n! k!} H^Q_{n,k} (q,\bq)  .   
	\end{align}	 
\end{remark}

\begin{remark}
The image of $E_\ell^-$ by $P_k\Ch$ reduces to zero for all $\ell > k$. Otherwise, it is one dimensional vector space. While $P_k\Ch(E_\ell^+) $ is always a one dimensional vector space.
\end{remark}

We conclude by giving the explicit sequence of eigenvalues of the integral operator 
$$|P_k\Ch|:= \left( (P_k\Ch)^* P_k\Ch \right)^{1/2}.$$

\begin{proposition}
 The functions $ \psi_{k,n}$; $n=0,1,2,\cdots,$ are eigenfunctions of the operator $|P_k\Ch |^2=(P_k\Ch )^*  P_k\Ch$ with $$\lambda_{k,n}:= \frac{\norm{\psi_{k,n}}^2}{\pi n!k!}$$ as corresponding eigenvalues.
\end{proposition}

\begin{proof}
	Notice first that the formal adjoint of $P_k\Ch$ is given by 
	\begin{align*} 
	(P_k\Ch )^* g (p) &= \int_{\Hq}  \overline{\mathcal{R}_{k}(q,p)} g(q) d\mu(q).
	\end{align*}
	Hence, from the expansion series of $\overline{\mathcal{R}_{k}(q,p)} = \mathcal{R}_{k}(\bp,\bq)$, 				
	we get				
	\begin{align*} 
	(P_k\Ch )^* g (p) 
	&= \sum_{m=0}^\infty \frac{  e^{-|p|^2}H^Q_{m,k-1} (p,\bp) 	 }{\pi m! k!} \int_{\Hq}   \overline{H^Q_{m,k} (q,\bq) } g(q) d\mu(q).
	\end{align*}
	so that
	\begin{align}\label{actionPCadj} 
	(P_k\Ch )^* (H^Q_{n,\ell} ) (p) 
&=  e^{-|p|^2}	H^Q_{n,k-1} (p,\bp)       \delta_{k,\ell} 
	.
	\end{align}
This combined with \eqref{actionPCpsig} yield 
	\begin{align}
(P_k\Ch )^*  P_k\Ch  \psi_{\ell,n} (p) 
&=     \frac{\varepsilon_{n+k-\ell}  d_k^{\ell,n} }{\pi (n+k-\ell)! k!}  (P_k\Ch )^*  H^Q_{n+k-\ell,k} (p) \nonumber
\\&=     \frac{\varepsilon_{n+k-\ell}  d_k^{\ell,n}  }{\pi (n+k-\ell)! k!} e^{-|p|^2}	H^Q_{n+k-\ell,k-1} (p,\bp) \nonumber
\\&=     \frac{\varepsilon_{n+k-\ell}   d_k^{\ell,n} }{\pi (n+k-\ell)! k!} \overline{\Ch	H^Q_{k,n+k-\ell} (p)} .\label{actiPCadjPC}  
\end{align}	
This shows in particular that for $\ell=k$, we have 
$$ |P_k\Ch |^2 \psi_{k,n} = \frac{\norm{\psi_{k,n}}_{\Hq}^2}{\pi n!k!} \psi_{k,n} = \lambda_{k,n} \psi_{k,n}.$$
\end{proof}	

\begin{corollary} The singular values of $P_k\Ch$ are given by 
	$$s_{n}^k	=  
	   \left( \frac{4^{n-1 } (k)!   }{3^{(n+k)/2}n!(k-n+1)!}
	\, {_2F_1}\left( \begin{array}{c} 1-n  , 1-n   \\ k-n+2\end{array}\bigg | \dfrac1 4 \right)\right)^{1/2}$$
	when $n\leq k+1$, and 
		$$s_{n}^k	=  
 \left(  \frac{4^{ k }(n-1 )!   }{3^{(n+k)/2} k! (n-1-k)!}
		\, {_2F_1}\left( \begin{array}{c} -  k  , - k  \\ n-k\end{array}\bigg | \dfrac1 4 \right) \right)^{1/2}$$ 
	otherwise.

\end{corollary}

\begin{remark}
The previous result shows in particular that the product of $k$-Bergman projector with the weighted  quaternionic Cauchy transform $\Ch$ and its complex analogue have the same singular values, up to multiplicative constant. {Subsequently, these singular values are asymptotically as 
$$s_{n}^k 
\sim
\left( \frac{4^{k}(n-1)!}{3^{n+k}n k! (n-1-k)!}\right) ^{1/2}$$ 
for $k$ is fixed and $n$ large enough.} 
\end{remark}

	\begin{remark}
		The previous result can be obtained using the integral representation of the operator $(P_k\Ch)^* P_k\Ch $, 
		\begin{align} \label{integralPChPch}
		(P_k\Ch)^* P_k\Ch f (p) = \int_{\Hq} \mathcal{S}_{k}(p,q) f(q) d\mu(q) 	
		\end{align}
		where the Schwartz kernel $\mathcal{S}_{k}$ is given by 
		\begin{align} 
		\mathcal{S}_{k}(p,q) &:=  
		\int_{\Hq} \overline{\mathcal{R}_{k}(\xi,p)} 
		\mathcal{R}_{k}(\xi,q)  d\mu(\xi) \nonumber
	\\&= e^{-|p|^2-|q|^2} \sum_{m=0}^\infty  \frac{H^Q_{m,k-1} (p,\bp)   H^Q_{m,k-1} (q,\bq)}{\pi m! k!} .  \label{kernelPChPch}
		\end{align}
		The last equality follows from  \eqref{expRk} and the orthogonality of $H^Q_{m,k}$.	
\end{remark}
	
\begin{remark}	
	For $k\geq 1$, the kernel $\mathcal{S}_{k}$ is closely connected to the reproducing kernel of $\mathcal{F}_{k-1,slice}^{2}$, 
	\begin{align} \label{kernelPChPchk1}
	\mathcal{S}_{k}(p,q) = \frac{e^{-|p|^2-|q|^2}}{k}  \mathcal{K}_{k-1}(\bp,\bq).
	\end{align}
\end{remark}

\end{document}